\documentclass[reqno]{amsart}

\usepackage{enumerate}
\usepackage{comment}
\usepackage{empheq}
\usepackage{amsmath}
\usepackage{amssymb}
\usepackage{array}
\usepackage{graphicx}
\usepackage{tikz}
\usepackage{pgf}
\usepackage{pgfpages}
\usepackage{comment}
\usetikzlibrary{decorations.pathreplacing}
\usetikzlibrary{calc}
\usetikzlibrary{decorations}
\usetikzlibrary{snakes}

\usetikzlibrary{matrix}
\tikzstyle{vertex}=[circle,thin,draw=black!100,fill=black!100, inner sep=0pt, minimum width=4pt]
\tikzstyle{vertexsm}=[circle,thin,draw=black!100,fill=black!100, inner sep=0pt, minimum width=1.2pt]
\tikzstyle{vertexnm}=[circle,thin,draw=black!100,fill=white!100, inner sep=0pt, minimum width=4pt]
\tikzstyle{vertexg}=[circle, draw=black, inner sep=1pt, style=densely dotted, minimum width=4pt]
\tikzstyle{vertexinf}=[circle,thin,draw=black!100,fill=white!100, inner sep=0pt, minimum width=10pt]
\tikzstyle{pedge}=[draw=black!100,-]
\tikzstyle{nedge}=[draw=black!100,densely dashed]
\tikzstyle{gedge}=[draw=black!100,densely dotted]

\newtheorem{theorem}{Theorem}
\newtheorem{proposition}[theorem]{Proposition}

\newtheorem{observation}[theorem]{Observation}

\begin{document}

\title[Bipartite complements of line graphs and $\lambda_3$]{A connection between the bipartite complements of line graphs and the line graphs with two positive eigenvalues}
\author{Lee Gumbrell}
\address{Department of Mathematics, Royal Holloway, University of
London, Egham Hill, Egham, Surrey, TW20 0EX, England, UK.}
\email{Lee.Gumbrell.2009@rhul.ac.uk}


\begin{abstract}
In $1974$ Cvetkovi\'{c} and Simi\'{c} showed which graphs $G$ are the bipartite complements of line graphs. In $2002$ Borovi\'{c}anin showed which line graphs $L\left(H\right)$ have third largest eigenvalue $\lambda_3\leq0$. Our first observation is that two of the graphs Borovi\'{c}anin found are the complements of two of the graphs found by Cvetkovi\'{c} and Simi\'{c}. Using the Courant-Weyl inequalities we show why this is and reprove the result of Borovi\'{c}anin, highlighting some features of the graphs found by both.
\end{abstract}

\maketitle

When related graphs appear for different reasons, it is important to understand why. As usual $L\left(H\right)$ denotes the line graph of a graph $H$, and the eigenvalues of (the adjacency matrix of) a graph on $n$ vertices are $\lambda_1\ge\lambda_2\ge\cdots\ge\lambda_n$. Also, $K_n$, $K_{m,n}$, $C_n$ and $P_n$ denote the complete graph on $n$ vertices, the complete bipartite graph on $m+n$ vertices, and the cycle and path on $n$ vertices, respectively. In $1974$ Cvetkovi\'{c} and Simi\'{c} showed in \cite{CS1974} the following result.

\begin{theorem}[see \cite{CS1974}, Theorem 8]\label{T:cvet+sim}
A graph $G$ is bipartite and the complement of a line graph if and only if $G$ is an induced subgraph of some of the graphs $CS_1$, $CS_2=CS_2\left(n\right)$ (with $n\geq0$) and $CS_3=CS_3\left(m,n,p\right)$ (with $p< n\leq m$; $p\geq0$, $m,n\geq1$) in Figure \ref{F:cvet+sim}.
\end{theorem}

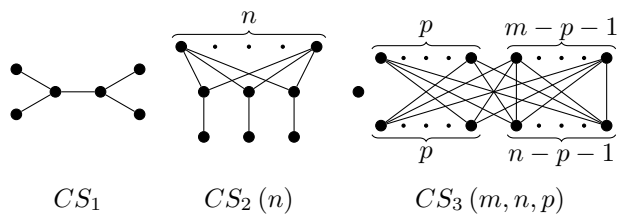
\begin{figure}[h]
\begin{center}
\begin{tabular}{ccc}
\begin{tikzpicture}[scale=0.6, auto] 
\foreach \pos/\name in
{{(0,0)/a},{(0.866,0.5)/b},{(0.866,-0.5)/c},{(-1,0)/aa},{(-1.866,0.5)/ab},{(-1.866,-0.5)/ac}}
\node[vertex] (\name) at \pos {};
\foreach \edgetype/\source/ \dest in {pedge/a/b,pedge/c/a,pedge/aa/ab,pedge/aa/ac,pedge/a/aa}
\path[\edgetype] (\source) -- (\dest);
\foreach \pos/\name in
{{(0,-1.65)/za},{(0,1.9)/zb}}
\node[] (\name) at \pos {};
\end{tikzpicture}
&
\begin{tikzpicture}[scale=0.6, auto] 
\foreach \pos/\name in
{{(0,0)/a},{(1,0)/b},{(-1,0)/c},{(0,-1)/e},{(1,-1)/f},{(-1,-1)/d},{(-1.5,1)/g},{(1.5,1)/h}}
\node[vertex] (\name) at \pos {};
\foreach \edgetype/\source/ \dest in {pedge/a/e,pedge/c/d,pedge/b/f,pedge/g/a,pedge/g/b,pedge/g/c,pedge/h/a,pedge/h/b,pedge/h/c}
\path[\edgetype] (\source) -- (\dest);
\draw[snake=brace] (-1.7,1.2)--(1.7,1.2);
\node[] at (0,1.65) {$n$};
\foreach \pos/\name in
{{(-0.75,1)/aa},{(0,1)/bb},{(0.75,1)/cc}}
\node[vertexsm] (\name) at \pos {};
\foreach \pos/\name in
{{(0,-1.65)/za},{(0,1.9)/zb}}
\node[] (\name) at \pos {};
\end{tikzpicture}
&
\begin{tikzpicture}[scale=0.6, auto] 
\foreach \pos/\name in
{{(-2.5,0.75)/a},{(-0.5,0.75)/b},{(0.5,0.75)/c},{(2.5,0.75)/d},{(-2.5,-0.75)/e},{(-0.5,-0.75)/f},{(0.5,-0.75)/g},{(2.5,-0.75)/h},{(-3,0)/i}}
\node[vertex] (\name) at \pos {};
\foreach \edgetype/\source/ \dest in {pedge/a/f,pedge/a/g,pedge/a/h,pedge/b/e,pedge/b/g,pedge/b/h,pedge/c/e,pedge/c/f,pedge/c/g,pedge/c/h,pedge/d/e,pedge/d/f,pedge/d/g,pedge/d/h}
\path[\edgetype] (\source) -- (\dest);
\draw[snake=brace] (-2.7,0.95)--(-0.3,0.95);
\node[] at (-1.5,1.4) {$p$};
\draw[snake=brace,mirror snake] (-2.7,-0.95)--(-0.3,-0.95);
\node[] at (-1.5,-1.4) {$p$};
\draw[snake=brace] (0.3,0.95)--(2.7,0.95);
\node[] at (1.5,1.4) {$m-p-1$};
\draw[snake=brace,mirror snake] (0.3,-0.95)--(2.7,-0.95);
\node[] at (1.5,-1.4) {$n-p-1$};
\foreach \pos/\name in
{{(-2,0.75)/aa},{(-1.5,0.75)/ab},{(-1,0.75)/ac},{(2,0.75)/ba},{(1.5,0.75)/bb},{(1,0.75)/bc},{(-2,-0.75)/ca},{(-1.5,-0.75)/cb},{(-1,-0.75)/cc},{(2,-0.75)/da},{(1.5,-0.75)/db},{(1,-0.75)/dc}}
\node[vertexsm] (\name) at \pos {};
\foreach \pos/\name in
{{(0,-1.65)/za},{(0,1.9)/zb}}
\node[] (\name) at \pos {};
\end{tikzpicture}
\\
$CS_1$ & $CS_2\left(n\right)$ & $CS_3\left(m,n,p\right)$
\end{tabular}
\caption{The three graphs from Theorem \ref{T:cvet+sim}.}
\label{F:cvet+sim}
\end{center}
\end{figure}

Later, Borovi\'{c}anin proved the following result.

\begin{theorem}[see \cite{B2002}, Theorem 3]\label{T:borov}
A connected line graph $L\left(H\right)$ has $\lambda_3\leq0$ if and only if $L\left(H\right)$ is an induced subgraph of some of the graphs $B_1$, $B_2$, $B_3=B_3\left(n\right)$ (with $n\geq0$) and $B_4=B_4\left(m,n,p\right)$ (with $p< n\leq m$; $p\geq0$, $m,n\geq1$) in Figure \ref{F:borov}.
\end{theorem}

\begin{figure}[h]
\begin{center}
\begin{tabular}{cccc}
\begin{tikzpicture}[scale=0.6, auto] 
\foreach \pos/\name in
{{(0,0)/a},{(-0.707,0.707)/b},{(0.707,0.707)/c},{(0.707,-0.707)/d},{(-0.707,-0.707)/e},{(-1.414,0)/f},{(1.414,0)/g}}
\node[vertex] (\name) at \pos {};
\foreach \edgetype/\source/ \dest in {pedge/a/b,pedge/a/c,pedge/a/d,pedge/a/e,pedge/b/c,pedge/b/e,pedge/b/f,pedge/c/d,pedge/c/g,pedge/d/e,pedge/d/g,pedge/e/f}
\path[\edgetype] (\source) -- (\dest);
\foreach \pos/\name in
{{(0,-1.65)/za},{(0,1.4)/zb}}
\node[] (\name) at \pos {};
\end{tikzpicture}
&
\begin{tikzpicture}[scale=0.6, auto] 
\foreach \pos/\name in
{{(0.688,0.5)/a},{(-0.261,0.809)/b},{(-0.851,0)/c},{(-0.261,-0.809)/d},{(0.688,-0.5)/e},{(1.554,0)/f},{(-1.561,0.809)/g},{(-2.151,0)/h},{(-1.561,-0.809)/i}}
\node[vertex] (\name) at \pos {};
\foreach \edgetype/\source/ \dest in {pedge/a/b,pedge/a/c,pedge/a/d,pedge/a/e,pedge/b/c,pedge/b/d,pedge/b/e,pedge/c/d,pedge/c/e,pedge/d/e,pedge/a/f,pedge/b/g,pedge/b/h,pedge/g/h,pedge/c/g,pedge/c/i,pedge/g/i,pedge/d/h,pedge/d/i,pedge/h/i,pedge/f/e}
\path[\edgetype] (\source) -- (\dest);
\foreach \pos/\name in
{{(0,-1.65)/za},{(0,1.4)/zb}}
\node[] (\name) at \pos {};
\end{tikzpicture}
&
\begin{tikzpicture}[scale=0.6, auto] 
\foreach \pos/\name in
{{(0,0)/a},{(0.32,0.809)/b},{(0.32,-0.809)/d},{(-0.710,0.809)/g},{(-1.3,0)/h},{(-0.710,-0.809)/i}}
\node[vertex] (\name) at \pos {};
\foreach \edgetype/\source/ \dest in {pedge/b/g,pedge/b/h,pedge/g/h,pedge/a/g,pedge/a/i,pedge/g/i,pedge/d/h,pedge/d/i,pedge/h/i}
\path[\edgetype] (\source) -- (\dest);
\draw (2.4,0) arc (0:360:1.2cm);
\node[] at (1.2,0) {$K_{n+3}$};
\foreach \pos/\name in
{{(0,-1.65)/za},{(0,1.4)/zb}}
\node[] (\name) at \pos {};
\end{tikzpicture}
&
\begin{tikzpicture}[scale=0.6, auto] 
\foreach \pos/\name in
{{(0,0)/a},{(-0.245,-0.7)/b},{(0.245,-0.7)/c},{(-1.2,-1.2)/d},{(1.2,-1.2)/e}}
\node[vertex] (\name) at \pos {};
\draw (0,0) arc (0:360:1.2cm);
\draw (2.4,0) arc (0:360:1.2cm);
\draw (-0.245,-0.7) arc (240:300:0.49cm);
\draw (-1.2,-1.2) arc (240:300:2.4cm);
\node[] at (-1.2,0) {$K_m$};
\node[] at (1.2,0) {$K_n$};
\node[] at (0,-1) {$\vdots$};
\node[] at (0.29,-1.19) {$p$};
\foreach \pos/\name in
{{(0,-1.65)/za},{(0,1.4)/zb}}
\node[] (\name) at \pos {};
\draw (0.58,-.83) node[rotate=61]{$\vdots$};
\draw (-0.58,-.83) node[rotate=-61]{$\vdots$};
\end{tikzpicture}
\\
$B_1$ & $B_2$ & $B_3\left(n\right)$ & $B_4\left(m,n,p\right)$
\end{tabular}
\caption{The four graphs from Theorem \ref{T:borov}. The large circles are used to denote a complete graph of that size.}
\label{F:borov}
\end{center}
\end{figure}
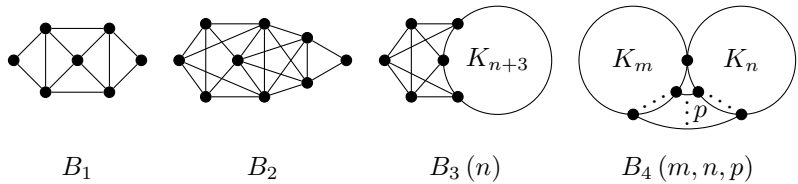

We make the immediate observation from these:

\begin{observation}\label{C:leecor}
For the graphs in Figures \ref{F:cvet+sim} and \ref{F:borov}, $B_3=\overline{CS_2}$, $B_4=\overline{CS_3}$ and $B_2 \supset \overline{CS_1}$.
\end{observation}
 
In this paper we explore why these graphs are related and in doing so we offer a new proof of Theorem \ref{T:borov}. Two important tools in this work are the Courant-Weyl inequalities and interlacing.

\begin{theorem}[Courant-Weyl inequalities, see \cite{PI1994}, Theorem 34.2.1] \label{T:courantweyl}
For two Hermitian $n\times n$ matrices $A$ and $B$ we have
\begin{center}
\begin{tabular}{ccccc}
$\lambda_i\left(A + B\right)$ &$\leq$& $\lambda_{i-j+1}\left(A\right) + \lambda_j\left(B\right)$ & &   $\left(i \geq j\right)$,\\
$\lambda_i\left(A + B\right)$ &$\geq$& $\lambda_{i-j+n}\left(A\right) + \lambda_j\left(B\right)$ & &  $\left(i \leq j\right)$.
\end{tabular}
\end{center}
\end{theorem}

\begin{theorem}[Cauchy \cite{Ca1829}; or see \cite{GR2001}, Theorem 9.1.1] \label{T:interlacing}
Let $G$ be an $n$-vertex graph with vertex set $V\left(G\right)$ and eigenvalues $\lambda_1 \geq \lambda_2 \geq \ldots \geq \lambda_n$. Also, let $H$ be the induced graph on $V\left(G\right)\backslash\left\{v\right\}$ obtained from $G$ by deleting the vertex $v$ and its incident edges.
Then the eigenvalues $\mu_1 \geq \mu_2 \geq \ldots \geq \mu_{n-1}$ of $H$ interlace with those of $G$; that is
\begin{displaymath}
\lambda_1 \geq \mu_1 \geq \lambda_2 \geq \mu_2 \geq \ldots \geq \mu_{n-1} \geq \lambda_n.
\end{displaymath}
\end{theorem}

We also need a couple of results about line graphs. Note that the graphs in Figure \ref{F:borov} are drawn in such a way that the structure given in Theorem \ref{T:linegraph} (i) is fairly easy to spot.

\begin{theorem} \label{T:linegraph}
\begin{enumerate}[(i)]
	\item \textup{(\cite{Ha1969}, Theorem 8.4)} A graph is a line graph if and only if its edges can be partitioned in such a way that every edge is in one clique and no vertex is in more than two cliques.
	\item \textup{(\cite{GR2001}, Lemma 8.6.2)} Line graphs have all their eigenvalues in the interval $\left[-2,\infty\right)$.
\end{enumerate}
\end{theorem}

Cvetkovi\'{c} showed the following result using the second of the Courant-Weyl inequalities and Theorem \ref{T:linegraph}(ii). 

\begin{theorem}[see \cite{C1982}, Theorem 2] \label{T:cvetcw}
If $G$ is the complement of a line graph then $\lambda_2\left(G\right)\leq 1$.
\end{theorem}

Using the same method and the symmetry of the eigenvalues of a bipartite graph we show that there can be even more structure in the spectrum of a line graph when it has a bipartite complement.

\begin{proposition}\label{P:lee}
If a line graph $L\left(H\right)$ has a bipartite complement, then $\lambda_3\left(L\left(H\right)\right)\leq0$.
\end{proposition}

\begin{proof}
Let $G=\overline{L\left(H\right)}$. The spectrum of a bipartite graph is symmetric around $0$, meaning that $\lambda_1=-\lambda_n$, $\lambda_2=-\lambda_{n-1}$, and so forth. This fact and Theorem \ref{T:cvetcw} tell us that if $G$ is bipartite then we also have that $\lambda_{n-1}\left(G\right)\geq-1$.

Take the second of the Courant-Weyl inequalities in Theorem \ref{T:courantweyl} and let $A=L\left(H\right)$ and $B=G$ so that $A+B=K_n$. Also let $i=2$ and $j=n-1$ so that $i-j+n=3$. The spectrum of the complete graph is well known and has $\lambda_i\left(K_n\right)=-1$ for $i=2,\ldots,n$. Bringing these together we get
\begin{eqnarray*}
  \lambda_2\left(K_n\right) &\geq& \lambda_3\left(L\left(H\right)\right) + \lambda_{n-1}\left(G\right) \\
  0 &\geq& \lambda_3\left(L\left(H\right)\right) + \lambda_{n-1}\left(G\right) + 1. 
\end{eqnarray*}
Then for $G$ bipartite we deduce that $\lambda_3\left(L\left(H\right)\right)\leq0$, as required.
\end{proof}

We now prove the two directions of Theorem \ref{T:borov} separately, starting with the reverse.

\begin{proposition}\label{P:borovbackwards}
If a line graph $L\left(H\right)$ is an induced subgraph of one of the graphs $B_1$, $B_2$, $B_3$ or $B_4$, then $\lambda_3\left(L\left(H\right)\right)\leq0$.
\end{proposition}

\begin{proof}
The graphs $B_1$ and $B_2$ have only a finite number of vertices so we can easily find their eigenvalues and see that $\lambda_3\leq0$. Any subgraphs will also then have $\lambda_3\leq0$ by interlacing (Theorem \ref{T:interlacing}). Since $CS_2$ and $CS_3$ are bipartite and the complements of line graphs, Observation \ref{C:leecor} and Proposition \ref{P:lee} tell us that we must then have $\lambda_3\left(B_3\right)\leq0$ and $\lambda_3\left(B_4\right)\leq0$.
\end{proof}

\begin{proposition}\label{P:borovforwards}
If a connected line graph $L\left(H\right)$ has $\lambda_3\leq0$ then it is an induced subgraph of some of the graphs $B_1$, $B_2$, $B_3$ and $B_4$.
\end{proposition}

\begin{proof}
Let $G$ be the complement of $L\left(H\right)$, then $G$ is certainly either bipartite or non-bipartite. If $G$ is bipartite then it is the bipartite complement of a line graph, so by Theorem \ref{T:cvet+sim} it must be an induced subgraph of $CS_1$, $CS_2$ or $CS_3$. Therefore by Observation \ref{C:leecor} $L\left(H\right)$ must be an induced subgraph of $B_3$, $B_4$ or contained in $B_2$.

If $G$ is non-bipartite we have a bit more work to do. In this case we know that $G$ contains an odd cycle $C_n$ for some odd $n$ and that $\overline{C_n}$ must be in $L\left(H\right)$. For $n\geq7$, we have the path on $6$ vertices $P_6$ as a subgraph of $C_n$ and $\lambda_2\left(P_6\right)>1$ so by interlacing we also have $\lambda_2\left(C_n\right)>1$. Using the second of the Courant-Weyl inequalities in Theorem \ref{T:courantweyl} again, with $i=j=2$, $B=C_n$ and $A+B=K_n$ we get that $\lambda_n\left(\overline{C_n}\right)<-2$. This means that for odd $n\geq7$ $\overline{C_n}$ cannot be a line graph nor an induced subgraph of a line graph by Theorem \ref{T:linegraph}(ii). Furthermore, the complement of the cycle $C_5$ is $C_5$ itself, which has $\lambda_3>0$. The conclusion to all this is that in the complement of a line graph with $\lambda_3\leq0$ the only odd cycles we will find will be of length $3$.

We now know that $G$ contains a $K_3$ so that means that $L\left(H\right)$ contains $\overline{K_3}=3K_1$. To complete the proof we grow line graphs starting with $3K_1$, increasing the number of vertices. To grow line graphs we recall the structure in Theorem \ref{T:linegraph}(i) and consider the cliques. At each step we can either
\begin{itemize}
  \item expand a clique to the next larger one
  \item expand two non-adjacent cliques and have them share the new vertex
  \item attach a single vertex pendent path (a $K_2$) to any vertex currently in only one clique.
\end{itemize}
Once we have done all these we must then consider adding to each one an edge between any two vertices in only one clique each (in effect, adding in another $K_2$), except between any two of the original three non-adjacent vertices. At each step we discard any graphs with $\lambda_3>0$. We do not add in extra isolated vertices or grow two separate graphs (that is, start growing from one vertex of $3K_1$ and then start growing from another without connecting them) as any resulting connected graphs can be grown using the method above without giving unnecessary extra disconnected graphs along the way.  

When the growing has reached $12$ vertices we pause and have a look at the graphs we have. There are $37$ non-isomorphic line graphs on $12$ vertices that contains an induced $3K_1$ and have $\lambda_3\leq0$: $B_3\left(4\right)\cup2K_1$; $B_3\left(5\right)\cup K_1$; $B_4\left(m,n,p\right)\cup2K_1$ for the triples $\left(6,5,p\right)$, $\left(7,4,p\right)$, $\left(8,3,p\right)$, $\left(9,2,p\right)$, $\left(10,0,0\right)$ with varying appropriate $p$; and $B_4\left(m,n,p\right)\cup K_1$ for the triples $\left(6,6,p\right)$, $\left(7,5,p\right)$, $\left(8,4,p\right)$, $\left(9,3,p\right)$, $\left(10,2,p\right)$ again with varying $p$.

We know that ultimately we want a connected line graph. Any attempts to connect these graphs either by adding in edges or growing them further whilst keeping an induced $3K_1$ will result in a graph with $\lambda_3>0$ (Figure 3 in \cite{B2002} has some subgraphs with $\lambda_3>0$ that are very easy to spot in this process). Looking at the graphs with $11$ vertices or fewer we see the graphs $B_1$ and $B_2$, and all their subgraphs (with potentially some extra isolated vertices -- if the graph without them contains an induced $3K_1$ then we can safely ignore them) along with some other subgraphs of $B_3$ and $B_4$ with one or two isolated vertices but no others.
\end{proof}

This proof highlights the fact that there are only finitely many connected line graphs with $\lambda_3\leq0$ with non-bipartite complements, but infinitely many with bipartite ones. By counting the non-isomorphic graphs that appear in the growing process we see that there are in fact only $19$ connected line graphs with $\lambda_3\leq0$ and non-bipartite complement. Another (easier) way to spot this is to count the number of non-isomorphic non-bipartite induced subgraphs there are of $\overline{B_1}$ and $\overline{B_2}$; there are $24$ of these but $5$ have disconnected complements. Knowing that there are only finitely many from the original proof of Theorem \ref{T:borov} in advance helps us know that the growing process used in the proof above will actually terminate.


\begin{thebibliography}{10}

\bibitem{B2002} B.~Borovi{\'c}anin, {Line graphs with exactly two positive eigenvalues}, \emph{Publ. Inst. Math. (Beograd)}, \textbf{72(86)} (2002), 39--47.
    
\bibitem{Ca1829} A.L.~Cauchy, {Sur l'\'{e}quation \`{a} l'aide de laquelle on d\'{e}termine les in\'{e}galiti\'{e}s s\'{e}culaires des mouvements des plan\`{e}tes}, Oeuvres compl\`{e}tes, Ii\`{e}eme S\'{e}rie, \textbf{9}, 174--195, Gauthier-Villars, 1829.
    
\bibitem{C1982} D.~Cvetkovi\'{c}, {On graphs whose second largest eigenvalue does not exceed $1$}, \emph{Publ. Inst. Math. (Beograd)}, \textbf{31(45)} (1982), 15--20.

\bibitem{CRS2004} D.~Cvetkovi\'{c}, P.~Rowlinson, S.~Simi\'{c}, \emph{Spectral generalizations of line graphs: on graphs with least eigenvalue $-2$}, London Mathematical Society Lecture Note Series \textbf{314}, Cambridge University Press, Cambridge, 2004.

\bibitem{CS1974} D.~Cvetkovi\'{c}, S.~Simi\'{c}, {Some remarks on the complement of a line graph}, \emph{Publ. Inst. Math. (Beograd)}, \textbf{17(31)} (1974), 37--44.
    
\bibitem{GR2001} C.~Godsil, G.~Royle, \emph{Algebraic Graph Theory}, Springer, New York, 2001.

\bibitem{Ha1969} F.~Harary, \emph{Graph Theory}, Addison-Wesley, 1969.
    
\bibitem{PI1994} V.~Prasolov, S.~Ivanov, \emph{Problems and theorems in linear algebra}, American Mathematical Society, 1994.


\end{thebibliography}
\end{document}